\documentclass[12pt]{amsart}
\usepackage{amssymb,amsmath,tabularx}
\usepackage{amsthm}
\usepackage{exscale}
\usepackage{graphicx,booktabs}
\usepackage[all]{xy}
\usepackage{pstricks,pst-node,pst-plot}
\usepackage[bookmarks=true]{hyperref}
\numberwithin{equation}{section}
\newtheorem{thm}[equation]{Theorem}

\newtheorem*{thankyou}{Acknowledgement}
\newtheorem{cor}[equation]{Corollary}
\newtheorem{hypo}[equation]{Hypothesis}
\newtheorem{lem}[equation]{Lemma}
\newtheorem{prop}[equation]{Proposition}
\theoremstyle{definition}
\newtheorem{defn}[equation]{Definition}
\newtheorem{eg}[equation]{Example}

\newcommand{\Stab}{\operatorname{Stab}}

\newcommand{\sgn}{\operatorname{sgn}}

\newcommand{\res}{\operatorname{res}}

\newcommand{\Ext}{\operatorname{\mathrm{Ext}}}

\renewcommand{\mod}{\operatorname{mod\,}}

\begin{document}
\title[]{The Varieties for some Specht Modules}
\author{Kay Jin Lim}

\maketitle
\begin{abstract} J. Carlson introduced the cohomological and rank variety
for a module over a finite group algebra. We give a general form for
the largest component of the variety for the Specht module for the
partition $(p^p)$ of $p^2$ restricted to a maximal elementary
abelian $p$-subgroup of rank $p$. We determine the varieties of a
large class of Specht modules corresponding to $p$-regular
partitions. To any partition $\mu$ of $np$ of not more than $p$
parts with empty $p$-core we associate a unique partition
$\Phi(\mu)$ of $np$, where the rank variety of the restricted Specht
module $S^\mu{\downarrow_{E_n}}$ to a maximal elementary abelian
$p$-subgroup $E_n$ of rank $n$ is $V_{E_n}^\sharp(k)$ if and only if
$V_{E_n}^\sharp(S^{\Phi(\mu)})=V_{E_n}^\sharp(k)$. In some cases
where $\Phi(\mu)$ is a 2-part partition, we show that the rank
variety $V_{E_n}^\sharp(S^\mu)$ is $V^\sharp_{E_n}(k)$. In
particular, the complexity of the Specht module $S^\mu$ is $n$.
\end{abstract}

\section{Introduction}

Over last forty years, group cohomology has been studied
extensively, especially its interaction with modular representation
theory. Carlson \cite{JC1} introduced the varieties for modules over
group algebras. Various results have been published which relate the
properties of the algebraic variety of a module and the structure of
the module itself.

Hemmer and Nakano studied the support varieties for permutation
modules and Young modules \cite{DHDN} over the symmetric groups. The
varieties for most Specht modules remain unknown. A partition is
made up of $(p\times p)$-blocks if each part is a multiple of $p$
and each part is repeated a multiple of $p$ times. The VIGRE
research group in Georgia (2004) conjectured that the variety for
the Specht module corresponding to a partition $\mu$ is the variety
of the defect group of the block in which the Specht module lies
unless and only unless the partition is made up of $(p\times
p)$-blocks.

It is not difficult to verify the conjecture when $\mu$
has $p$-weight strictly less than $p$. In the first part of this
paper, we study the variety for the Specht module corresponding to
the partition $(p^p)$. In the latter part, we verify the conjecture
for a large class of partitions of $p^2$ and $np$ for some positive
integer $n$, where they are not made up of $(p\times p)$-blocks.

We organize this paper as follows. A brief introduction to the
varieties for modules and representation theory of symmetric groups
is given in Section \ref{notations}, the standard texts are
\cite{DB} II, \cite{GJ1} and \cite{GJAK}. We also set up notations
in this section. Our main results are stated in Section \ref{main
results}. In Section \ref{p by p}, in the case $p$ is odd, we show
that the dimension of the rank variety for the Specht module
$S^{(p^p)}$ is $p-1$ . This therefore gives the complexity of the
Specht module. In \cite{JC2}, Carlson gives an upper bound for the
degree of the projectivized rank variety for a module over an
elementary abelian $p$-group. Later in this section, we show that
the radical ideal corresponding to the largest component of
$S^{(p^p)}{\downarrow_{E_p}}$ is generated by a single non-zero
polynomial, where a general formula is given. This general
polynomial shows that the degree of the projectivized rank variety
for the module $S^{(p^p)}{\downarrow_E}$ is non-zero and divisible by $(p-1)^2$.

In Section \ref{not p by p}, we determine the varieties of a large
class of Specht modules corresponding to $p$-regular partitions. For
each partition $\mu$ of $np$ with no more than $p$ parts which has
empty $p$-core, we associate a unique partition $\Phi(\mu)$ of $np$
such that each part of the partition is a multiple of $p$ and with
the property that the rank variety $V_{E_n}^\sharp(S^\mu)$ is
$V^\sharp_{E_n}(k)$ if and only if
$V_{E_n}^\sharp(S^{\Phi(\mu)})=V^\sharp_{E_n}(k)$. In some cases
where $\Phi(\mu)$ is a 2-part partition, we show that the rank
variety for the Specht module $S^\mu$ restricted to a maximal
elementary abelian $p$-subgroup $E$ of the symmetric group
$\mathfrak{S}_{np}$ of rank $n$ is precisely the rank variety for
the trivial module. In this case, the complexity of the Specht
module is $n$. At the end of this section, we give a complete list
of the varieties for the modules $S^\mu{\downarrow_{E_3}}$ where
$\mu$ are partitions of 9, over $p=3$.

\section{Background Materials and Notations}\label{notations}

Throughout this paper $p$ is a prime and $k$ is an algebraically
closed field of characteristic $p$. Let $G$ be a finite group whose
order is divisible by $p$ and $M$ be a finitely generated
$kG$-module. Carlson \cite{JC1} introduced the cohomological variety
$V_G(M)$ for the module $M$. Avrunin and Scott show that the variety
$V_G(M)$ has a stratification \cite{DB}
$$V_G(M)=\bigcup_{E\in\mathcal{E}(G)}\res^\ast_{G,E}V_E(M)$$ where
$\mathcal{E}(G)$ is a set of representatives for the conjugacy
classes of elementary abelian $p$-subgroups of $G$ and
$\res^\ast_{G,E}:V_E(k)\to V_G(k)$ is the map induced by the
restriction map $\res_{G,E}:\Ext^\bullet_{kG}(k,k)\to
\Ext^\bullet_{kE}(k,k)$. The complexity of the module $M$ is
precisely the dimension $\dim V_G(M)$ of the variety $V_G(M)$. Since
the map $\res^\ast_{G,E}$ is a finite map, we have $$\dim
V_G(M)=\max_{E\in\mathcal{E}(G)}\{\dim \res^\ast_{G,E}
V_E(M)\}=\max_{E\in \mathcal{E}(G)}\{\dim V_E(M)\}$$

Let $E$ be an elementary abelian group of order $p^n$ with
generators $g_1,\ldots,g_n$ and $M$ be a finitely generated
$kE$-module. For a generic affine point $0\neq
\alpha=(\alpha_1,\ldots,\alpha_n)\in k^n$, the Jordan type of the
restriction $M{\downarrow_{k\langle u_\alpha\rangle}}$ is called the
generic Jordan type of the module $M$, where
$u_\alpha=1+\sum_{i=1}^n \alpha_i(g_i-1)$ (see \cite{JC1} and
\cite{EFJPAS}). Generic Jordan type is compatible with direct sum
and tensor product of two $kE$-modules, where the tensor product is
given by the diagonal $E$-action (Proposition 4.7 \cite{EFJPAS}).
The stable generic Jordan type of $M$ is the generic Jordan type of
$M$ modulo projective direct summands. The rank variety
$V_E^\sharp(M)$ of the $kE$-module $M$ is the set
$$\{0\}\cup \{0\neq \alpha\in k^n\,|\,\text{$M{\downarrow_{k\langle
u_\alpha\rangle}}$ is not $k\langle u_\alpha\rangle$-free}\}$$ It is
isomorphic to the cohomological variety of the module $M$ (via the
Frobenius map when $p$ is odd) \cite{GALS}. Note that $M$ is not
generically free if and only if its rank variety $V^\sharp_E(M)$ is
$V_E^\sharp(k)$.

\begin{prop}\label{summary of properties of varieties} Let $M,N$ be $kG$-modules.
\begin{enumerate}
\item [$\mathrm{(i)}$] $V_G(M)=\{0\}$ if and only if $M$ is projective $kG$-module.

\item [$\mathrm{(ii)}$] $V_G(M\oplus N)=V_G(M)\cup V_G(N)$ and $V_G(M\otimes_k N)=V_G(M)\cap V_G(N)$ with the diagonal $G$-action.

\item [$\mathrm{(iii)}$] If $p\not|\dim_k M$, then $V_G(M)=V_G(k)$.

\item [$\mathrm{(iv)}$]  Suppose that $M$ is an indecomposable $kG$-module in a block
with defect group $D$. Then $V_G(M)=\res^\ast_{G,D}V_D(M)$. Let
$|G|=p^as$ with $\gcd(s,p)=1$. Then $p^a/|D|$ divides $\dim_kM$.
Furthermore, if
$$\gcd\left (\frac{|D|\dim_k M}{p^a},p\right )=1$$ then
$V_G(M)=\res^\ast_{G,D}V_D(k)$.

\item [$\mathrm{(v)}$] If $H$ be a subgroup of $G$, then $V_H(M)=
\left (\res^\ast_{G,H}\right )^{-1}V_G(M)$.

\item [$\mathrm{(vi)}$] The variety $V_G(M)$ is a homogeneous subvariety
of $V_G(k)$. If $M$ is indecomposable, then the projectivized
variety $\overline{V_G(M)}$ is connected when endowed with the
Zariski topology.

\item [$\mathrm{(vii)}$]  Suppose further that $G$ is an elementary abelian $p$-group
of rank $n$ and $r=\dim V_G(M)$. Then $p^{n-r}$ divides $\dim_k M$.
\end{enumerate}
\end{prop}

The standard texts for the representation of symmetric groups are
\cite{GJ1} and \cite{GJAK}. Let $\mathfrak{S}_m$ denote the
symmetric group on $m$ letters. For each $s\leq n$, let $E_s$ be the
elementary abelian subgroup of $\mathfrak{S}_{np}$ generated by the
$p$-cycles $((i-1)p+1,\ldots,ip)$ for $1\leq i\leq s$. A partition
$\mu$ of $m$ is a sequence of positive integers
$(\mu_1,\mu_2,\ldots,\mu_s)$ such that $\mu_1\geq \mu_2\geq
\ldots\geq \mu_s$ and $\sum_{i=1}^s\mu_i=m$; in this case we write
$|\mu|=m$. The Young subgroup $\mathfrak{S}_\mu$ of $\mathfrak{S}_m$
is
$$\mathfrak{S}_{\{1,\ldots,\mu_1\}}\times
\mathfrak{S}_{\{\mu_1+1,\ldots,\mu_1+\mu_2\}}\times\ldots\times
\mathfrak{S}_{\{\mu_1+\ldots+\mu_{s-1}+1,\ldots,\mu_1+\ldots+\mu_s\}}$$
The permutation module $M^\mu\cong
k_{\mathfrak{S}_\mu}{\uparrow^{\mathfrak{S}_m}}$ is the
$k\mathfrak{S}_m$-module $k$-spanned by all $\mu$-tabloids $\{t\}$
where $\mathfrak{S}_m$ acts by permuting the numbers assigned to the
nodes of $t$. The Specht module $S^\mu$ is the submodule of $M^\mu$
$k$-spanned by the $\mu$-polytabloids $e_t=\sum_{\sigma\in
C_t}\sgn(\sigma)\sigma \{t\}$ where $C_t$ is the column stabilizer
of the $\mu$-tableau $t$. It has dimension given by the hook formula
$$\dim_k S^\mu=\frac{m!}{\displaystyle{\prod_{i,j\in[\mu]}h_{ij}}}$$ where
$h_{ij}$ is the hook length of the $(i,j)$-node of the Young diagram
$[\mu]$ of $\mu$. The number of skew $p$-hooks removed to obtain the
$p$-core $\widetilde{\mu}$ of $\mu$ is precisely the number of hook
lengths of $\mu$ divisible by $p$, we write $m_\mu$ for this common
number. A defect group $D_\mu$ of the block containing the Specht
module $S^\mu$ is a Sylow $p$-subgroup of the symmetric group
$\mathfrak{S}_{m_\mu p}$.

The Young module $Y^\mu$ corresponding to the partition $\mu$ is an
indecomposable direct summand of $M^\mu$ with the property that each
indecomposable summand of $M^\mu$ is isomorphic to some $Y^\lambda$
with $\lambda\unrhd\mu$ by the lexicographic ordering (Theorem 3.1
of \cite{GJ2}) and $Y^\mu$ occurs with multiplicity one. Each
$Y^\mu$ has a Specht filtration and the filtration multiplicity is
well-defined \cite{SD}. The permutation module $M^{(r-k,k)}$
corresponding to a two part partition $(r-k,k)$ has a Specht
filtration with factors
$$S^{(r)},S^{(r-1,1)},\ldots, S^{(r-k,k)}$$ reading from the top
(Theorem 17.13 of \cite{GJ1}). So any Young module $Y^\lambda$
appearing in the decomposition of $M^{(r-k,k)}$ satisfies
$\lambda=(r-s,s)$ for some $s\leq k$ and has multiplicity one.

\begin{prop}\label{summary of symmetric groups representation}\
\begin{enumerate}
\item [(i)]\textup{[The Branching Theorem ($\S 9$ of \cite{GJ1})]} Let $\mu$ be a partition of $m$
and $\Omega(\mu)$ be the set consisting of all partitions obtained
by removing a node from $\mu$, then the $k\mathfrak{S}_{m-1}$-module
$S^\mu{\downarrow_{\mathfrak{S}_{m-1}}}$ has a filtration with
Specht factors $S^\lambda$ one for each $\lambda\in\Omega(\mu)$.

\item[(ii)]\textup{[Nakayama's Conjecture]} Let $\mu,\lambda$ be
partitions of $m$. Then the Specht modules $S^\mu,S^\lambda$ lie in
the same block if and only if $\widetilde{\mu}=\widetilde{\lambda}$.
So any $k\mathfrak{S}_m$-module which has a filtration with Specht
factors decomposes into direct summands according to $p$-cores.

\item[(iii)] Let $C_p$ be the cyclic group of order $p$ and $J_i$
be the Jordan block of size $0\leq i\leq p$. Then there is an
extension of indecomposable $kC_p$-modules
$$0\to J_i\to J_a\oplus J_b\to J_j\to 0$$ with $a\geq b$ if and only if
$\max\{i,j\}\leq a\leq p$, $0\leq b\leq \min\{i,j\}$ and $a+b=i+j$.

\item[(iv)] If $\mu$ is a partition of $np$ with non-empty
$p$-core, then $S^\mu{\downarrow_{E_n}}$ is generically free.

\item [(v)] Let $\mu'$ be the conjugate of a partition $\mu$ of $m$. Then
$V_{\mathfrak{S}_m}(S^{\mu'})=V_{\mathfrak{S}_m}(S^\mu)$. In
particular, $V^\sharp_E(S^{\mu'})=V^\sharp_E(S^\mu)$ for any
elementary abelian $p$-subgroup $E$ of $\mathfrak{S}_m$.
\end{enumerate}
\end{prop}
\begin{proof} (iv) By Proposition \ref{summary of properties of varieties} (iv),
$V_{\mathfrak{S}_{np}}(S^\mu)=\res^\ast_{\mathfrak{S}_{np},D_\mu}
V_{D_\mu}(S^\mu)$ where $D_\mu$ is a Sylow $p$-subgroup of
$\mathfrak{S}_{m_\mu p}$. Since $m_\mu<n$, any maximal elementary
abelian $p$-subgroup of $D_\mu\leq \mathfrak{S}_{m_\mu p}$ has rank strictly less than $n$. So
$$\dim V_{E_n}(S^\mu)\leq \dim V_{\mathfrak{S}_{np}}(S^\mu)=\dim
V_{D_\mu}(S^\mu)\leq \max_{E\leq D_\mu}\{ \dim V_E(S^\mu)\}<n$$
where $E$ runs through all elementary abelian $p$-subgroups of
$D_\mu$ by the Quillen Stratification Theorem.

(v) Note that $S^{\mu'}\otimes_k S^{(1^m)}\cong (S^\mu)^\ast$ (see
8.15 \cite{GJ1}). By 5.7.3 and 5.8.5 of \cite{DB} II, we have
$V_{\mathfrak{S}_m}(S^{\mu'})=V_{\mathfrak{S}_m}(S^{\mu'})\cap
V_{\mathfrak{S}_m}(k)=V_{\mathfrak{S}_m}\left ((S^\mu)^\ast\right
)=V_{\mathfrak{S}_m}(S^\mu)$.
\end{proof}

\section{Main Results}\label{main results}

The VIGRE research group in Georgia (2004) conjectured that the
variety for the Specht module corresponding to a partition $\mu$ is
the variety of the defect group where the Specht module lies in if
and only if the partition is not made up of $(p\times p)$-blocks. Let
$W_i(S^{(p^p)})$ be the union of all irreducible components of
$V^\sharp_{E_p}(S^{(p^p)})\subseteq V^\sharp_{E_p}(k)$ of dimension
$i$. In Section \ref{p by p}, we prove the following theorem.

\begin{thm}\label{theorem p by p} Suppose that $p$ is an odd prime and $\mu=(p^p)$ is the $(p\times p)$-partition.
\begin{enumerate}
\item [$\mathrm{(i)}$] $\dim V_{\mathfrak{S}_{p^2}}(S^\mu)=p-1$.

\item [$\mathrm{(ii)}$] We have $I(W_{p-1}(S^{(p^p)}))=\sqrt{f}=\langle f\rangle$ where
$$\hspace{.5cm}f(x_1,\ldots,x_p)= (x_1\ldots
x_p)^{p-1}\widetilde{f}+\sum_{i=1}^{p} {x_1}^{n(p-1)}\ldots
\widehat{{x_i}^{n(p-1)}}\ldots{x_p}^{n(p-1)}$$ for some homogeneous
polynomial $\widetilde{f}\in
k[x_1,\ldots,x_p]^{(\mathbb{F}_p^\times)^p\rtimes \mathfrak{S}_p}$
and positive integer $n$ where ${x_1}^{n(p-1)}\ldots
\widehat{{x_i}^{n(p-1)}}\ldots {x_p}^{n(p-1)}$ is the product of all
${x_j}^{n(p-1)}$'s, $1\leq j\leq p$, except the term
${x_i}^{n(p-1)}$.
\end{enumerate}
\end{thm}

In Section \ref{not p by p}, we determine the varieties of a large
class of Specht modules that are not made up of $(p\times
p)$-blocks. We associate to any partition $\mu=(\mu_1,\ldots,\mu_s)$
of $np$ not more than $p$ parts with empty $p$-core a unique
partition $\Phi(\mu)$ of $np$ such that
$\Phi(\mu)=(n_1p,n_2p,\ldots,n_rp)$ and $r\leq s$.

\begin{hypo}\label{hypothesis} Suppose that $\mu$ is a partition of $np$ not more
than $p$ parts with empty $p$-core and one of the following
conditions hold.
\begin{enumerate}
\item [(H1)] The prime $p$ is odd, $n=p$ and $\Phi(\mu)$ is a $2$-part partition $(p^2-mp,mp)$ of $p^2$ for some $1\leq m<
p/2$.
\item [(H2)] The prime $p$ is odd and $\Phi(\mu)$ is a $2$-part partition $(np-\varepsilon p,\varepsilon p)$
such that $n\not\equiv 2(\mod p)$ and $\varepsilon\in\{1,2\}$.
\item [(H3)] $\Phi(\mu)=(np)$.
\item [(H4)] $p=2$ and $\Phi(\mu)$ is the partition $(2n-2,2)\neq (2,2)$ or $(2n-4,4)\neq (4,4)$.
\end{enumerate}
\end{hypo}

\begin{thm}\label{theorem not p by p} If $\mu$ is a partition satisfying Hypothesis
\ref{hypothesis}, then $V^\sharp_{E_n}(S^\mu)=V^\sharp_{E_n}(k)$. In
this case, the complexity of $S^\mu$ is $n$.
\end{thm}

\section{The Variety for the Specht Module $S^{(p^p)}$}\label{p by
p}

\begin{thm} Suppose that the Specht module $S^\mu$ corresponding to
a partition $\mu$ of $n$ lies inside a block with abelian defect
$D_\mu$. We have
$V_{\mathfrak{S}_n}(S^\mu)=\res^\ast_{\mathfrak{S}_n,D_\mu}V_{D_\mu}(k)$
and the complexity of the Specht module $S^\mu$ is exactly the
$p$-weight of $\mu$.
\end{thm}
\begin{proof} Let $n!=p^ab$ with $\gcd(a,b)=1$ and $m_\mu$ be the $p$-weight of $\mu$. Since the number of hook
lengths of $[\mu]$ divisible by $p$ is at least as many as $m_\mu$,
we have
$$\mathbb{Z}\ni\frac{|D_\mu|}{p^a}(\dim_k S^\mu)_p=\frac{p^{m_\mu}}{p^a}\frac{p^a}{\prod_{(i,j)\in
[\mu]}(h_{ij})_p}=1$$ By Proposition \ref{summary of properties of
varieties} (iv),
$V_{\mathfrak{S}_n}(S^\mu)=\res^\ast_{\mathfrak{S}_n,D_\mu}V_{D_\mu}(k)$.
The complexity of the $S^\mu$ is exactly the $p$-rank of the defect
group $D_\mu$, i.e., $m_\mu$.
\end{proof}

For $p=2$, let $F$ be the maximal abelian subgroup of
$\mathfrak{S}_4$ generated by $(12)(34)$ and $(13)(24)$. The module
$S^{(2^2)}{\downarrow_F}$ is isomorphic to $k\oplus k$, so $\dim
V_{\mathfrak{S}_4}(S^{(2^2)})=2$. For the rest of Section \ref{p by
p}, $p$ is an odd prime.

\subsection{Some vanishing ideals}

Given the symmetric group $\mathfrak{S}_m$ of degree $m$, the group
$\left (\mathbb{F}_p^\times\right )^m\rtimes \mathfrak{S}_m$ is
defined by the group actions
$(\beta,\sigma)(\beta',\sigma')=(\beta\cdot\sigma(\beta'),\sigma\sigma')$
where
$\sigma(\beta')=(\beta'_{\sigma^{-1}(1)},\ldots,\beta'_{\sigma^{-1}(m)})$
for all $\sigma,\sigma'\in\mathfrak{S}_m$ and
$\beta=(\beta_1,\ldots,\beta_m),\beta'=(\beta'_1,\ldots,\beta'_m)\in(\mathbb{F}_p^\times)^m$.
This group acts on the polynomial ring $k[x_1,\ldots,x_m]$ via
$$(\beta,\sigma){x_i}:=\beta_{\sigma(i)}x_{\sigma(i)}$$ for all
$1\leq i\leq m$ and $(\beta,\sigma)\in \left
(\mathbb{F}_p^\times\right )^m\rtimes \mathfrak{S}_m$. We think of
the action in two stages, first by $(\mathbb{F}_p^\times)^m$ and
then followed by $\mathfrak{S}_m$.
%
%
%
%
\begin{lem}\label{some factors of f}
Let $p$ be an odd prime, $m\geq 3$, $f\in k[x_1,\ldots,x_m]^{\left
(\mathbb{F}_p^\times \right )^m\rtimes A_m}$ and $x_i$ be a factor
of $f$ for some $1\leq i\leq m$. Then ${x_1}^{p-1}\ldots
{x_m}^{p-1}$ is a factor of $f$.
\end{lem}
\begin{proof} Since $f\in \left (k[x_1,\ldots,x_m]^{\left (\mathbb{F}_p^\times
\right )^m}\right )^{A_m}=k[{x_1}^{p-1},\ldots,{x_m}^{p-1}]^{A_m}$,
it follows that if $x_i$ divides $f$ then ${x_i}^{p-1}$ divides $f$.
As $f$ is invariant under the action of $A_m$, all ${x_j}^{p-1}$'s
divide $f$. All factors ${x_j}^{p-1}$ are pairwise coprime, so the
product divides $f$.
\end{proof}

\begin{lem}\label{f is fixed by some subgroup}
Suppose that $m\geq 3$, $G=\left (\mathbb{F}_p^\times\right
)^m\rtimes \mathfrak{S}_m$ and $f$ is a non-zero polynomial in
$k[x_1,\ldots,x_m]$ such that the ideal $\langle f\rangle$ generated
by $f$ satisfies the property $G\langle f\rangle=\langle f\rangle$,
i.e., for all $g\in G$, we have $gf\in \langle f\rangle$. Then $f$
is either fixed by the subgroup $\left (\mathbb{F}_p^\times\right
)^m\rtimes A_m$ of $G$ or divisible by $x_1\ldots x_m$. In the first
of these cases, $f$ is either fixed or negated by the action of $G$.
\end{lem}
\begin{proof} Write $R=k[x_1,\ldots,x_m]$ and $f=\sum a_{n_1\ldots
n_m}{x_1}^{n_1}\ldots{x_m}^{n_m}$ with $a_{n_1\ldots n_m}\in k$.
Suppose that for each $(\beta,\sigma)\in G$, we have
${(\beta,\sigma)}f=hf$ for some polynomial $h\in R$. Since the
action of $(\beta,\sigma)$ on each monomial of $f$ is multiplication
by some $\beta_i$'s and permuting $x_j$'s, the highest degree of a
monomial appearing in $f$ is at least as large as the highest degree
of a monomial appearing in $hf$. So, $h=c(\beta,\sigma)$ lies in
$k$. Let $\mathbf{1}=(1,\ldots,1)$ in $(\mathbb{F}_p^\times)^m$. Let
$\beta(i)$ be the element in $\left (\mathbb{F}_p^\times\right )^m$
such that its $i$th coordinate is $\beta_i$ and $1$ elsewhere. It
is clear that $c(\mathbf{1},\sigma)^{-1}=c(\mathbf{1},\sigma^{-1})$
and
$$c(\beta,\sigma)=c(\beta(1),1)c(\beta(2),1)\ldots c(\beta(m),1)c(\mathbf{1},\sigma)$$ for all $\beta\in \left
(\mathbb{F}_p^\times\right )^m$ and $\sigma \in \mathfrak{S}_m$.

It suffices to examine the action of transpositions on $f$ to
determine $c(\mathbf{1},\sigma)$. For any transposition $\sigma$,
$f=\sigma^2 f={c(\mathbf{1},\sigma)^2}f$, so
$c(\mathbf{1},\sigma)\in \{\pm 1\}$. This shows that
$c(\mathbf{1},\xi)=1$ for all $\xi\in A_m$. Suppose that there
exists some monomial \[a_{n_1\ldots
n_m}{x_1}^{n_1}\ldots{x_m}^{n_m}\] involved in $f$ with $n_i=0$ for
some $1\leq i\leq m$. By the action of $\beta(i)$ on $f$ we have
$a_{n_1\ldots n_m}=c(\beta(i),1)a_{n_1\ldots n_m}$, i.e.,
$c(\beta(i),1)=1$. For any $1\leq j\leq m$, let $\sigma \in A_m$
such that $\tau (i)=j$, then $\tau f=f$ gives $c(\beta(j),1)=1$. So
$(\mathbb{F}_p^\times)^m\rtimes A_m$ acts trivially on $f$ in this
case. Otherwise, every monomial involved in $f$ is divisible by
$x_1$ and so $f$ is divisible by $x_1\ldots x_m$.
\end{proof}

For any fixed $1\leq i\leq m$, define $x_1\ldots \widehat{x_i}\ldots
x_m$ as the product of all $x_j$'s where $1\leq j\neq i\leq m$. We
write $I_i$ for the radical ideal of the polynomial ring
$k[x_1,\ldots,x_m]$ generated by $x_1\ldots \widehat{x_i}\ldots x_m$
and $x_i$ .

\begin{lem}\label{vanishing ideal of union} For $m\geq 2$, the vanishing
ideal of the union of the varieties $V(I_i)$ where $1\leq i\leq m$ is $(
\{x_1\ldots\widehat{x_i}\ldots x_m\,|\, 1\leq i\leq m\})$.
\end{lem}
\begin{proof} The algebraic variety given by the union of all $V(I_i)$'s
consists of all planes at least two coordinates having value 0, which
is precisely the algebraic variety defined by the ideal $J$
generated by $x_1\ldots\widehat{x_i}\ldots x_m$ for all $1\leq i\leq
m$. Furthermore, $J$ is a radical ideal, so both sides coincide
using Hilbert's Nullstellensatz.
\end{proof}

\begin{prop}\label{polynomial defines certain variety} Let $m\geq 2$
and $I$ be a radical ideal of $k[x_1,\ldots,x_m]$ generated by a
single polynomial $f$ such that $V(f)\cap
V(x_i)=V(x_1\ldots\widehat{x_i}\ldots x_m,x_i)$ for any $1\leq i\leq
m$. Then
$$f(x_1,\ldots,x_m)=x_1\ldots
x_m\widetilde{f}+\sum_{i=1}^ma_i{x_1}^{n_{1i}}\ldots
{x_m}^{n_{mi}}$$ such that $\widetilde{f}\in k[x_1,\ldots,x_m]$ and
\begin{enumerate}
\item [$\mathrm{(i)}$] $n_{ii}=0$ for all $1\leq i\leq m$,
\item [$\mathrm{(ii)}$] $0\neq a_i\in k$ for all $1\leq i\leq m$,
\item [$\mathrm{(iii)}$] $n_{ij}>0$ for all $i\neq j$.
\end{enumerate}
\end{prop}
\begin{proof} The given hypothesis implies $\bigcup V\left
(x_1\ldots\widehat{x_i}\ldots x_{m},x_i\right ) \subseteq V(f)$. By
Lemma \ref{vanishing ideal of union}, $f$ lies inside the ideal
$(\{x_1\ldots\widehat{x_i}\ldots x_m\,|\,1\leq i\leq m\})$, i.e.,
$$f(x_1,\ldots,x_m)=\sum_{i=1}^m f_ix_1\ldots\widehat{x_i}\ldots
x_m$$ for some $f_1,\ldots, f_m\in k[x_1,\ldots,x_m]$. It is not
difficult to see that we can rewrite $f$ as follows
$$f(x_1,\ldots,x_m)=x_1\ldots
x_m\widetilde{f}+\sum_{i=1}^ma_i(x_1,\ldots,x_m){x_1}^{n_{1i}}\ldots
{x_m}^{n_{mi}}\eqno(\ast)$$ such that
\begin{enumerate}
\item [(i)] $n_{ij}>0$ if $j\neq i$ and $n_{ii}=0$ for all $1\leq i\leq m$,

\item [(ii)] $\widetilde{f}, a_i\in k[x_1,\ldots,x_m]$ for each $1\leq i\leq m$,

\item [(iii)]  for all $1\leq i,j\leq m$, $x_j$ does not divide $a_i$ and $a_i$ is independent of the variable $x_i$.

\end{enumerate} For each $1\leq i\leq m$, the equation
$ f(x_1,\ldots,x_m)\left |_{x_i=0}\right
.=a_i{x_1}^{n_{1i}}\ldots{x_m}^{n_{mi}}$ implies
$$V(a_i,x_i)=V(a_i)\cap V(x_i)\subseteq V(f)\cap
V(x_i)=V(x_1\ldots\widehat{x_i}\ldots x_m,x_i)$$ By Hilbert's
Nullstellensatz, we have $(x_1\ldots\widehat{x_i}\ldots x_m,
x_i)\subseteq (a_i,x_i)$. Suppose that $x_1\ldots \widehat{x_i}\ldots
x_m=a_iw_i+x_iv_i$ for some $w_i,v_i\in k[x_1,\ldots, x_m]$. We write
$w_i=u_i+x_iu_i'$ such that $u_i$ is independent of the variable
$x_i$. The equation $$x_1\ldots \widehat{x_i}\ldots
x_m-a_iu_i=a_ix_iu_i'+x_iv_i$$ implies that $a_ix_iu_i'+x_iv_i=0$ as
the left-hand side is independent of $x_i$. So we have
$x_1\ldots\widehat{x_i}\ldots x_m=a_iu_i$.

Since $x_j$ does not divide $a_i$ for all $j\neq i$, we have
$u_i/x_1\ldots\widehat{x_i}\ldots x_m$ is a polynomial. Comparing
degrees, we have $a_i\in k$. Note that $a_i\neq 0$; otherwise,
\[(x_1\ldots\widehat{x_i}\ldots x_m,x_i)\subseteq (x_i)\] Now we have
all the desired properties.
\end{proof}

\subsection{Proof of Theorem \ref{theorem p by p}}

Let $E=E_p$. Now we are ready to prove Theorem \ref{theorem
p by p} (i).

\begin{proof}[Proof of Theorem \ref{theorem p by p}(i)] Recall that
$\mu=(p^p)$. By the Branching Theorem, we have
$S^\mu{\downarrow_{\mathfrak{S}_{p^2-1}}}\cong S^\tau$ where
$\tau=(p^{p-1},p-1)$. Note that $\tau$ has $p$-core
$\widetilde{\tau}=(p,1^{p-1})$ and the corresponding block has
defect group $D_\tau=E_{p-2}\leq \mathfrak{S}_{p^2-1}$. Using the
hook formula
$$\frac{|D_\tau|(\dim_k
S^\tau)_p}{|\mathfrak{S}_{p^2-1}|_p}=\frac{p^{p-2}p}{p^{p-1}}=1$$ so
$V_{\mathfrak{S}_{p^2-1}}(S^\mu)=V_{\mathfrak{S}_{p^2-1}}(S^\tau)=\res^\ast_{\mathfrak{S}_{p^2-1},E_{p-2}}
V_{E_{p-2}}(k)$ by Proposition \ref{summary of properties of
varieties} (iv). This also shows that $\dim
V_{\mathfrak{S}_{p^2-1}}(S^\mu)=p-2$ and so $\dim
V_{\mathfrak{S}_{p^2}}(S^\mu)\leq p-1$: if $\dim
V_{\mathfrak{S}_{p^2}}(S^\mu)=p$, then $V_E(S^\mu)=V_E(k)$ and so
$\dim V_{\mathfrak{S}_{p^2-1}}(S^\mu)=\dim V_{E_{p-1}}(S^\mu)=p-1$.
On the other hand, let $r=\dim V_E(S^\mu)\leq \dim
V_{\mathfrak{S}_{p^2}}(S^\mu)$, using Proposition \ref{summary of
properties of varieties} (vii), $p^{p-r}$ divides $(\dim_k
S^\mu)_p=p$, this implies $p-1\leq r$. So $\dim
V_{\mathfrak{S}_{p^2}}(S^\mu)=p-1$.
\end{proof}

The quotient group
$G=N_{\mathfrak{S}_{p^2}}(E)/C_{\mathfrak{S}_{p^2}}(E)\cong
(\mathbb{F}_p^\times)^p\rtimes \mathfrak{S}_p$ acts on $k^p$ via
$$(\beta,\sigma)\alpha=(\beta_{\sigma(1)}\alpha_{\sigma(1)},\ldots,\beta_{\sigma(p)}\alpha_{\sigma(p)})$$
for all $\alpha=(\alpha_1,\ldots,\alpha_p)\in k^p$. Let $J$ be the
Jacobson radical of $kE$. For each $1\leq i\leq p$ and
$b\in\mathbb{F}_p^\times$, we have ${g_i}^b-1\equiv b(g_i-1)(\mod
J^2)$. Lemma 6.4 of \cite{JC1} shows that $G$ acts on the rank
variety $V_E^\sharp(M)$ of a $kE$-module $M$ and hence on the
vanishing ideal $I(V_E^\sharp(M))$ of $V_E^\sharp(M)$ where the
action is given by
$(\beta,\sigma)x_i=\beta_{\sigma(i)}x_{\sigma(i)}$.

\begin{lem}\label{variety for restricted p by p 2}
Let $\tau=(p^{p-1},p-1)$ and $E_{p-1}\leq \mathfrak{S}_{p^2-1}$.
Then the rank variety $V_{E_{p-1}}^\sharp(S^{\tau})\subseteq
k^{p-1}$ is the union of all hyperplanes $V(x_i)\subseteq k^{p-1}$
where $1\leq i\leq p-1$.
\end{lem}
\begin{proof}
We knew that
$V_{\mathfrak{S}_{p^2-1}}(S^{\tau})=\res^\ast_{\mathfrak{S}_{p^2-1},E_{p-2}}V_{E_{p-2}}(k)$.
Apply Proposition \ref{summary of properties of varieties} (v), we
have $V_{E_{p-1}}(S^\tau)=\left
(\res^\ast_{\mathfrak{S}_{p^2-1},E_{p-1}}\right )^{-1}
\res^\ast_{\mathfrak{S}_{p^2-1},E_{p-2}} V_{E_{p-2}}(k)$.
So the rank variety $V^\sharp_{E_{p-1}}(S^\tau)$ lies in the union
of hyperplanes $\bigcup_{i=1}^{p-1} V(x_i)$. Since the rank variety contains
$V(x_{p-1})$ so it contains $V(x_i)$'s where $1\leq i\leq p-1$ given
by the action of $\mathfrak{S}_{p-1}$ on the $p-1$ coordinates of
$V^\sharp_{E_{p-1}}(S^\tau)$.
\end{proof}

Since $k[x_1,\ldots,x_m]$ is a unique factorization domain, prime
ideals of height one are principal. Let $W_{i}(S^{(p^p)})$ be the
union of all irreducible components of
$V_E^\sharp(S^{(p^p)})\subseteq k^p$ of dimension $i$. Since
$W_{p-1}(S^{(p^p)})$ has codimension 1, by the previous remark, it
is defined by a single polynomial in variables $x_1,\ldots,x_p$. By
Lemma \ref{variety for restricted p by p 2},
$$W_{p-1}(S^{(p^p)})\cap V(x_i)\subseteq V_E^\sharp(S^{(p^p)})\cap
V(x_i)=\bigcup_{j\neq i} V(x_j,x_i)$$ and $\dim \left
(W_{p-1}(S^{(p^p)})\cap V(x_i)\right )=p-2$, so
$W_{p-1}(S^{(p^p)})\cap V(x_i)$ contains one of the irreducible
varieties $V(x_j,x_i)$. Since $W_{p-1}(S^{(p^p)})$ is invariant
under the action of $G$, its intersection with $V(x_i)$ contains all
of $V(x_j,x_i)$ where $1\leq j\neq i\leq p$. We have proved the
following lemma.

\begin{lem}\label{variety of largest component}
For any $1\leq i\leq p$, we have $W_{p-1}(S^{(p^p)})\cap
V(x_i)=\bigcup_{j\neq i}V(x_j,x_i)$.
\end{lem}

\begin{proof}[Proof of Theorem \ref{theorem p by p}(ii)] By Lemma \ref{variety of largest component} and
Proposition \ref{polynomial defines certain variety}, $f$ has the
form
$$f(x_1,\ldots,x_p)=x_1\ldots
x_p\widetilde{f}'+\sum_{i=1}^pa_i{x_1}^{n_{1i}}\ldots{x_p}^{n_{pi}}$$
such that $\widetilde{f}'\in k[x_1,\ldots, x_p]$, $n_{ij}>0$ for all
$j\neq i$, $n_{ii}=0$ for all $1\leq i\leq p$ and $a_i\in k^\times $
for all $1\leq i\leq p$. The group $ \left
(\mathbb{F}_p^\times\right )^p\rtimes \mathfrak{S}_p$ acts on the
radical ideal generated by $f$. If $f$ is divisible by $x_1\ldots
x_p$, then $V(f)\cap V(x_1)$ contains $V(x_1)$, which contradicts
Lemma \ref{variety of largest component}. By Lemma \ref{f is fixed
by some subgroup}, $f$ is fixed by $\left (\mathbb{F}_p^\times\right
)^p\rtimes \mathfrak{S}_p$ up to a sign and fixed by the subgroup
$S=\left (\mathbb{F}_p^\times\right )^p\rtimes A_p$.

We write $w_i$ for ${x_1}^{n_{1i}}\ldots{x_p}^{n_{pi}}$. Let
$\tau=(12\ldots p)\in A_p$ be the $p$-cycle. For any $1\leq j\leq
p$, we have $\tau^{j-1}f=f$. So
$$(\tau^{j-1}\widetilde{f}'-\widetilde{f}')x_1\ldots
x_p=(a_jw_j-a_1\tau^{j-1}w_1)+\text{terms divisible by $x_j$}$$
where both $w_j, \tau^{j-1}w_1$ are independent of $x_j$. We must
have $a_j=a_1$ and $w_j=\tau^{j-1}w_1$. Let $a=a_1=\ldots=a_p\neq
0$. For any $\beta=(\beta_1,\ldots,\beta_p)\in
(\mathbb{F}^\times_p)^p$, let $\beta(j)$ be the element in
$(\mathbb{F}^\times_p)^p$ such that its $j$th coordinate is
$\beta_j$ and $1$ elsewhere. Note that $\beta(j)f=f$. So
$$(\beta_j(\beta(j)\widetilde{f}')-\widetilde{f}')x_1\ldots
x_p=\sum_{i\neq j}(a-a{\beta_j}^{n_{ji}})w_i$$ For $i\neq j$, $x_i$
divides the left-hand term and all $w_r$'s such that $r\neq i,j$ on
the right-hand side. So $a-a\beta^{n_{ji}}=0$, i.e.,
${\beta_j}^{n_{ji}}=1$. This shows that for $i\neq j$ each $n_{ji}$ is divisible
by $p-1$.

Note that $V(1/a f)=V(f)$, so we may assume that $f$ has the form
$$f(x_1,\ldots,x_p)=x_1\ldots x_p\widetilde{f}+
\sum_{i=0}^{p-1}\tau^i({x_2}^{n_2(p-1)}\ldots
{x_p}^{n_p(p-1)})\eqno(\ast)$$ where
$w_j=\tau^{j-1}{x_2}^{n_2(p-1)}\ldots{x_p}^{n_p(p-1)}$. The sum and
the term $x_1\ldots x_p\widetilde{f'}$ in ($\ast$) are clearly
invariant under the group $S$. By Lemma \ref{some factors of f},
$x_1\ldots x_p\widetilde{f}'={x_1}^{p-1}\ldots
{x_p}^{p-1}\widetilde{f}$ for some homogeneous polynomial
$\widetilde{f}$ fixed by $S$. For any $j\neq 2$, let $(2j)f=\epsilon
f$ where $\epsilon\in\{\pm 1\}$. So
$$((2j)\widetilde{f}-\epsilon\widetilde{f})x_1\ldots x_p=(\epsilon
aw_j-a(2j)w_2)+\text{terms divisible by $x_j$}$$ gives us
$w_j=(2j)w_2$. Comparing the power of $x_{j+1}$, we have $n_2=n_j$.
So $n_2=\ldots=n_p$.
\end{proof}

Let $V\subseteq \mathbb{P}^{n-1}(k)$ be a projective variety of
dimension $m$. The intersection of $V$ and a generic linear subspace
$W$ of $\mathbb{P}^{n-1}$ of dimension $n-m-1$ is a union of $s$
points for some fixed positive integer $s$. This positive integer is
the degree of the projective variety $V$. In case $\dim W<n-m-1$,
then $V\cap W=\varnothing$.

\begin{cor}\label{degree divisible by p-1 square} For $p\geq 3$,
the degree of the projectivized variety $\overline{V_E^\sharp \left
(S^{(p^p)}\right )}$ is non-zero and divisible by $(p-1)^2$.
\end{cor}
\begin{proof}
Let $W\subseteq k^p$ be a generic linear subspace of dimension 2.
For any component $W_i(S^{(p^p)})$ such that $0\leq i\leq p-2$, we
have $W_i(S^{(p^p)})\cap W=\{0\}$. So $$V_E^\sharp(S^{(p^p)})\cap
W=W_{p-1}(S^{(p^p)})\cap W=\text{$s$ lines passing through the
origin}$$ where $s$ is the degree of the homogeneous variety
$W_{p-1}(S^{(p^p)})$, i.e., the degree of a homogeneous polynomial
$f$ defining $W_{p-1}(S^{(p^p)})$ such that $\sqrt{f}=\langle
f\rangle$. So $s$ is non-zero and divisible by $(p-1)^2$.
\end{proof}

\section{The Variety for the Specht Module $S^{\mu}$ where $\mu$ is $p$-regular}\label{not p by p}

\subsection{The permutation modules}

Let $m,n$ be two non-negative integers and $m=\ldots+m_1p+m_0$,
$n=\ldots+n_1p+n_0$ be their $p$-adic expansions. The number $m$ is
$p$-contained in $n$ if $m_i\leq n_i$ for all $i\geq 0$. In this
case we write $m\subseteq_p n$. Note that $m\subseteq_p n$ if and
only if $\binom{n}{m}\not\equiv 0(\mod p)$. Corollaries
\ref{decomposition for p square} and \ref{decomposition for multiple
of p} rely on Theorem 3.3 of \cite{AH}.

\begin{thm}[3.3 of \cite{AH}] The Young module $Y^{(r-s,s)}$ is a direct summand of the
permutation module $M^{(r-m,m)}$ if and only if $m-s\subseteq_p
r-2s$.
\end{thm}

\begin{cor}\label{decomposition for p square}
For any integer $p< m\leq p^2/2$, we have a decomposition
$$M^{(p^2-m,m)}\cong M^{(p^2-m+p,m-p)}\oplus Q$$ where $Q$ has a
filtration with Specht factors $$S^{(p^2-m+p-1,m-p+1)},
S^{(p^2-m+p-2,m-p+2)},\ldots, S^{(p^2-m,m)}$$
\end{cor}
\begin{proof} It suffices to show that if $Y^{(p^2-s,s)}$ is a
direct summand of $M^{(p^2-m+p,m-p)}$, then it is also a direct
summand of $M^{(p^2-m,m)}$. Note that the trivial module $k\cong
Y^{(p^2)}$ is not a direct summand of $M^{(p^2-w,w)}$ for any
$0<w\leq p^2/2$. Let $0<s\leq m-p$ and $Y^{(p^2-s,s)}$ be a direct
summand of $M^{(p^2-m+p,m-p)}$, i.e.,
$$a_1p+a_0=m-p-s\subseteq_p p^2-2s=b_1p+b_0$$ where $0\leq
a_0,a_1,b_0,b_1\leq p-1$. Note that $a_0+(a_1+1)p$ is the $p$-adic
expansion of $m-s$. If $a_1+1>b_1$, then $a_1=b_1$. So
$(p^2-2s)-(m-p-s)=b_0-a_0< p$. On the other hand, since $s\leq m\leq
p^2/2$, we have $(p^2-2s)-(m-p-s)=p^2+p-(m+s)\geq p$, which is a
contradiction. This shows that $m-s\subseteq_p p^2-2s$.
\end{proof}

\begin{cor}\label{decomposition for multiple of p} Let $p$ be an odd
prime.
\begin{enumerate}
\item [$\mathrm{(i)}$] If $n\not\equiv 1,2 (\mod p)$, then $M^{(np-2p,2p)}\cong
M^{(np-p,p)}\oplus Q$;

\item [$\mathrm{(ii)}$] If $n\equiv 1(\mod p)$, then
$M^{(np-p,p)}\cong k\oplus N$ for some $k\mathfrak{S}_{np}$-module
$N$ and $M^{(np-2p,2p)}\cong N\oplus R$.
\end{enumerate} In both cases, $Q,R$ have filtrations with Specht factors
$$S^{(np-p-1,p+1)},S^{(np-p-2,p+2)},\ldots,S^{(np-2p,2p)}$$
\end{cor}
\begin{proof} Suppose that $0<s\leq p$ and let $Y^{(np-s,s)}$ be a direct summand of
$M^{(np-p,p)}$, i.e., $$a_0=p-s\subseteq_p
np-2s=\ldots+b_2p^2+b_1p+b_0$$ where $0\leq a_0,b_i<p$ for all
$i\geq 0$. Note that $b_0+2s\equiv 0(\mod p)$ and $a_0+p$ is the
$p$-adic expansion of $2p-s$. On the other hand, we have
$$p+1\leq p+s=a_0+2s\leq b_0+2s\leq (p-1)+2p=3p-1$$ So $b_0+2s=2p$.
Substituting into the $p$-adic expansion of $np-2s$, we see that
$n\equiv 2(\mod p)$ unless $b_1\neq 0$. This shows that
$2p-s\subseteq_p np-2s$ in both (i) and (ii). Suppose that
$n\not\equiv 1(\mod p)$, the equation
$$\binom{np}{2p}=\binom{np}{p}\cdot \frac{(np-2p+1)(np-2p+2)\ldots
(np-p)}{(p+1)(p+2)\ldots 2p}$$ implies that the trivial module $k$
is a direct summand of $M^{(np-p,p)}$ if and only if it is a direct
summand of $M^{(np-2p,2p)}$. So (i) is proved. If $n\equiv 1(\mod
p)$, then $k$ is a direct summand of $M^{(np-p,p)}$ but not of
$M^{(np-2p,2p)}$.
\end{proof}

\begin{lem}\label{generic jordan type of permutation modules}
Let $\mu=(n_1p,\ldots,n_sp)$ be a partition of $np$ such that
$n_1\geq n_2\geq \ldots \geq n_s>0$. Then the generic Jordan type of
the module $M^\mu{\downarrow_{E_n}}$ is $(p^a,1^b)$ where
$$a=\frac{\dim_k M^\mu-b}{p},\,\,b=\frac{n!}{n_1!n_2!\ldots n_s!}$$
\end{lem}
\begin{proof} The permutation module $M^\mu$ is isomorphic to the trivial
module induced from the Young subgroup $\mathfrak{S}_\mu$ to
$\mathfrak{S}_{np}$. Apply the Mackey decomposition formula
$$M^\mu{\downarrow_{E_n}}\cong k_{\mathfrak{S}_\mu}{\uparrow^{\mathfrak{S}_{np}}}{\downarrow_{E_n}}\cong
\bigoplus_{E_ng\mathfrak{S}_\mu}{}^gk{\downarrow_{{}^g\mathfrak{S}_\mu\cap
E_n}}{\uparrow^{E_n}}$$ Double coset representatives of
$E_n,\mathfrak{S}_\mu$ in $\mathfrak{S}_{np}$ correspond to the
orbits $\mathcal{O}_{E_n}(g_i)$ of the action of $E_n$ on the left
coset representatives $g_1,\ldots,g_m$ of $\mathfrak{S}_\mu$ in
$\mathfrak{S}_{np}$. The stabilizer $\Stab_{E_n}(g_i)$ of $g_i$
consists of precisely the elements $e\in E_n$ such that $eg_i\in
g_i\mathfrak{S}_\mu$. So $\Stab_{E_n}(g_i)=E_n$ if and only if
$\mathcal{O}_{E_n}(g_i)=\{g_i\}$ if and only if the $\mu$-tabloid
corresponding to $g_i$ is fixed by $E_n$. For each $\mu$-tabloid
$\{t\}$, denote by $R_i(t)$ the set consisting of integers in the
$i$th row of $\{t\}$. Take $\{t_0\}$ as the $\mu$-tabloid such that
$$R_i(t_0)=\left \{1+\sum^{i-1}_{j=0}\mu_j,\,\,
2+\sum^{i-1}_{j=0}\mu_j,\,\,\ldots\,\, ,\,\,
\sum^{i}_{j=0}\mu_j\right \}$$ So the $\mu$-tabloid corresponding to
$g_i$ is $\{g_it_0\}$. Note that the $\mu$-tabloids $\{t\}$ fixed by
$E_n$ are precisely those satisfying the property that for each
$1\leq i\leq n$ there exists some $1\leq u\leq s$ such that
$\{1+(i-1)p,\ldots,ip\}\subseteq R_u(t)$. So
$$k_{\mathfrak{S}_\mu}{\uparrow^{\mathfrak{S}_{np}}}{\downarrow_{E_n}}
\cong\left (\bigoplus_{\substack{E_ng\mathfrak{S}_\mu\\ E_n\cap
{}^g\mathfrak{S}_\mu<E_n}}k_{E_n\cap
{}^g\mathfrak{S}_\mu}{\uparrow^{E_n}}\right )\oplus \left
(\bigoplus_{\substack{E_ng\mathfrak{S}_\mu\\ E_n\cap
{}^g\mathfrak{S}_\mu=E_n}}k\right )$$ The summand $k_{E_n\cap
{}^g\mathfrak{S}_\mu}{\uparrow^{E_n}}$ is generically free if
$E_n\cap {}^g\mathfrak{S}_\mu<E_n$; otherwise, it has generic Jordan
type $(1)$, and there are precisely
$\frac{(n_1+\ldots+n_s)!}{n_1!\ldots n_s!}$ of them.
\end{proof}

\subsection{The map $\Phi$}

Every partition can be associated to a set of $\beta$-numbers and
represented by an abacus ($\S 2.7$ of \cite{GJAK}). For the rest of
our discussion, whenever we speak of an abacus of $\mu$, we mean the
abacus associated to the choice of $\beta$-numbers given by the
first column hook lengths of $\mu$. So the abacus of the partition
$\mu=(\mu_1,\ldots,\mu_s)$ has beads $\mu_i+(s-i)$ where $1\leq
i\leq s$. In the case $s\leq p$, the $p$-core $\widetilde{\mu}$ of
$\mu$ is empty if and only if for each $1\leq i\leq s$ there is a
unique $0\leq j_i\leq s-1$ such that $\mu_i+(s-i)\equiv j_i(\mod
p)$. If $\mu_i\not\equiv 0(\mod p)$ for some $1\leq i\leq s$, let
$1\leq b\leq s$ be the number such that $\mu_b\not\equiv 0(\mod p)$
and $\mu_i\equiv 0(\mod p)$ for all $b+1\leq i\leq s$. If we insert
the beads $\mu_s,\mu_{s-1}+1,\ldots,\mu_b+(s-b)$ successively into
the abacus and ignore the position of the bead in each runner, we
see the following
$$\xymatrix{\bullet &\ldots&\bullet&\circ&\ldots&\circ&\bullet_b}$$
where there are $s-b$ beads $\bullet$ on the left-hand side, more
than one $\circ$ in the middle and the last bead $\bullet_b$
corresponds to $\mu_b$. Since $\widetilde{\mu}=\varnothing$, there
is a unique number $1\leq a<b$ such that $\mu_a+(s-a)\equiv
\mu_b+(s-b)-1(\mod p)$, i.e., the bead immediately to the left of
$\bullet_b$.

\begin{hypo}\label{hypothesis for partition less than p parts}
Suppose that $\mu=(\mu_1,\ldots,\mu_s)$ is a partition with $s\leq
p$, $\widetilde{\mu}=\varnothing$ and there are unique numbers
$1\leq a< b\leq s$ with the properties:
\begin{enumerate}
\item [(i)] $\mu_i\equiv 0(\mod p)$ for all $b+1\leq i\leq s$ and
$\mu_b\not\equiv 0(\mod p)$,
\item [(ii)] $\mu_a-a\equiv \mu_b-1-b(\mod p)$.
\end{enumerate}
\end{hypo} Let $\mu$ be a partition satisfying Hypothesis \ref{hypothesis for partition less than p
parts},
$\eta=(\mu_1,\ldots,\mu_{b-1},\mu_b-1,\mu_{b+1},\ldots,\mu_s)$ and
$\Omega$ be the set consisting of all (proper) partitions $\mu(j)$
where $1\leq j\leq s+1$ such that $\mu(j)$ is the partition obtained
from $\eta$ by adding a node to the end of the $j$th row (assuming
$\mu_{s+1}=0$). We claim that $\mu(a),\mu(b)\in\Omega$ and there is
no $\mu(j)\in \Omega$ with empty $p$-core other than $\mu=\mu(b)$
and $\mu(a)$. Suppose that $\mu_{a-1}=\mu_a$, we have
$\mu_{a-1}+(s-(a-1))\equiv \mu_b-(s-b)(\mod p)$. This implies that
the $p$-core of $\mu$ is not empty. This contradiction shows that
$\mu_{a-1}>\mu_a$ and so $\mu(a)\in\Omega$. In the case $j=a$,
$\mu(a)_i+(s-i)=\mu_i+(s-i)$ for all $i\neq a,b$ and
\begin{align*}
\mu(a)_a+(s-a)&\equiv \mu_b+(s-b)(\mod p)\\
\mu(a)_b+(s-b)&\equiv \mu_a+(s-a)(\mod p)
\end{align*} So $\mu(a)$ has empty $p$-core because $\mu$ has empty
$p$-core. It is clear that $\mu=\mu(b)\in\Omega$.

For $1\leq j\leq s\leq p$ such that $j\neq a,b$, we have
$$\mu(j)_b+(s-b)\equiv \mu(j)_a+(s-a)(\mod p)$$ So the $p$-core of
$\mu(j)$ is not empty. In the case $j=s+1$, there are at most $p+1$
beads in the abacus of $\mu(s+1)$. It is clear that
$$\mu(s+1)_a+((s+1)-a)\equiv \mu_b+(s-b)\equiv
\mu(s+1)_b+((s+1)-b)(\mod p)$$ So $\mu(s+1)$ has non-empty $p$-core
unless $s=p$ and $\mu_b+(s-b)\equiv 0(\mod p)$, i.e., $\mu_b\equiv
b(\mod p)$. In this case, we must have $\mu_p\equiv 0(\mod p)$: if
not, then $b=p$ and $\mu_b\equiv 0(\mod p)$, contradicts to the
hypothesis of $\mu$. However there are two beads corresponding to
$\mu(p+1)_{p+1}$ and $\mu(p+1)_{p}$ lying in the first runner of the
abacus, so $\mu(p+1)$ has non-empty $p$-core.

Fix a positive integer $n$. Let $\Lambda$ be the set consisting of
all partitions $\mu=(\mu_1,\ldots,\mu_s)$ of $np$ with $s\leq p$ and
$\widetilde{\mu}=\varnothing$. We define a map $\phi:\Lambda\to
\Lambda$ as follows. $$\phi(\mu)=\left \{\begin{array}{ll}
\mu(a)&\text{if $\mu$ satisfies Hypothesis \ref{hypothesis for
partition less than p parts}}\\ \mu&\text{otherwise}
\end{array}\right .$$

\begin{eg}\label{an example which satisfies the main hypothesis}
Let $p$ be an odd prime, $\mu=(u,v,2^m)$ with $0\leq m\leq p-2$ and
$\widetilde{\mu}=\varnothing$; for instance, taking $u\equiv
p-m-1(\mod p)$ and $v\equiv p-m+1(\mod p)$. If $m>0$, then $b=m+2$
and $a=2$. So $\phi(\mu)=(u,v+1,2^{m-1},1)$. Now we follow the
procedure for $\phi(\mu)$, we get $b=m+2$ and $a=1$. So
$\phi^2(\mu)=(u+1,v+1,2^{m-1})$. By induction on the integer $m$, we
see that $\phi^{2m}(\mu)=(U,V)$ with $U\equiv p-1(\mod p)$, $V\equiv
1(\mod p)$ and $U+V=|\mu|$. So $\phi^{2m+1}(\mu)=(U+1,V-1)$, and
indeed $\phi^t(\mu)=(U+1,V-1)$ for all $t\geq 2m+1$ by our
definition.
\end{eg}

\begin{defn} Let $\mu$ be a partition not more than $p$ parts with
empty $p$-core. There is a positive integer $t(\mu)$ such that for
all integers $t\geq t(\mu)$ we have
$\phi^t(\mu)=\phi^{t(\mu)}(\mu)$. We define
$\Phi(\mu)=\phi^{t(\mu)}(\mu)$.
\end{defn}

Recall from Section \ref{notations} the definition of generic Jordan
type and stable generic Jordan type of a finitely generated
$kE$-module $M$ where $E$ is an elementary abelian $p$-group. We
have $V^\sharp_{E_n}(M)=V^\sharp_{E_n}(k)$ if and only if $M$ is not
generically free. For each $1\leq i\leq p$, we define $n_M(i)$ as
the number of Jordan blocks of size $i$ in the generic Jordan type
of $M$. In the case $M=S^\mu{\downarrow_{E_n}}$ where $\mu$ is a
partition of $np$, we write $n_\mu(i)$ for
$n_{S^\mu{\downarrow_{E_n}}}(i)$. Let $N$ be another finitely
generated $kE$-module. If $n_M(i)=n_N(p-i)$ for all $1\leq i\leq
p-1$, we say that $M,N$ is a pair of modules of complementary stable
Jordan type.

\begin{lem}\label{complementary sgjt} Suppose that $\mu=(\mu_1,\ldots,\mu_s)$
is a partition of $np$ with $1\leq s\leq p$ and
$\widetilde{\mu}=\varnothing$. Let $1\leq a<b\leq s$ be the unique
numbers such that
\begin{enumerate}
\item [(i)] $\mu_i\equiv 0(\mod p)$ for all $b+1\leq i\leq s$ and
$\mu_b\not\equiv 0(\mod p)$,
\item [(ii)] $\mu_a-a\equiv \mu_b-1-b(\mod p)$.
\end{enumerate} Then $\phi(\mu)$ has
empty $p$-core and
$S^\mu{\downarrow_{E_n}},S^{\phi(\mu)}{\downarrow_{E_n}}$ is a pair of
modules of complementary stable Jordan type. In particular,
$V^\sharp_{E_n}(S^\mu)=V^\sharp_{E_n}(k)$ if and only if
$V^\sharp_{E_n}\left (S^{\Phi(\mu)}\right )=V^\sharp_{E_n}(k)$.
\end{lem}

\begin{proof} Let $\eta$ and $\Omega$ as before. Since there are precisely two
partitions $\mu,\phi(\mu)$ in $\Omega$ with empty $p$-core, using
Proposition \ref{summary of symmetric groups representation} (i),
(ii) and (iv), we get
$$S^\eta{\uparrow^{\mathfrak{S}_{np}}}{\downarrow_{E_n}}\cong
Q\oplus \text{generically free direct summand}$$ where $Q$ is a
direct summand with factors $S^\mu{\downarrow_{E_n}}$ and
$S^{\phi(\mu)}{\downarrow_{E_n}}$, and Specht modules not in the
principal block contribute to the generically free summand. Since
$$\dim V_{E_n}(S^\eta{\uparrow^{\mathfrak{S}_{np}}})\leq \dim
V_{\mathfrak{S}_{np}}(S^\eta{\uparrow^{\mathfrak{S}_{np}}})=\dim
\res^\ast_{\mathfrak{S}_{np},\mathfrak{S}_{np-1}}V_{\mathfrak{S}_{np-1}}(S^\eta)\leq
n-1$$ the module
$S^{\eta}{\uparrow^{\mathfrak{S}_{np}}}{\downarrow_{E_n}}$ is
generically free. Hence
$S^\mu{\downarrow_{E_n}},S^{\phi(\mu)}{\downarrow_{E_n}}$ is a pair
of modules of complementary stable Jordan type. Since
$S^{\phi(\mu)}{\downarrow_{E_n}}$ has stable generic Jordan type
either the same or complementary to that of
$S^\mu{\downarrow_{E_n}}$, we get the second assertion.
\end{proof}

\begin{lem}\label{generic Jordan type case 1} Let $n\geq 2$ be a
positive integer and $\mu=(np-p,p)$. If $p$ is an odd prime, then
$n-2\leq n_\mu(1)\leq n+1$. If $p=2$, then $n_\mu(1)=n-2$.
\end{lem}
\begin{proof} Let $p$ be an odd prime. Note that $S^{(np)}{\downarrow_{E_n}}\cong k_{E_n}$ has generic Jordan type $(1)$
and $S^{(np-1,1)}{\downarrow_{E_n}}$ has stable generic Jordan type
$(p-1)$, which is complementary to $(1)$. The permutation module
$M^{(np-p,p)}$ decomposes into $Q\oplus S$ where
$S{\downarrow_{E_n}}$ is a generically free $kE_n$-module and $Q$
has three Specht factors $S^{(np)},S^{(np-1,1)}$ and $S^{(np-p,p)}$
reading from the top. The stable generic Jordan type for
$M^{(np-p,p)}{\downarrow_{E_n}}$ is given in Lemma \ref{generic
jordan type of permutation modules} as $(1^n)$. Suppose that $A$ is
the quotient $Q/S^{(np-p,p)}$. We want to figure out the possible
stable generic Jordan types of $A{\downarrow_{E_n}}$. Considering
the short exact sequence $0\to S^{(np-1,1)}{\downarrow_{E_n}}\to
A{\downarrow_{E_n}}\to k\to 0$ and using Proposition \ref{summary of
symmetric groups representation} (iii), we see that
$A{\downarrow_{E_n}}$ is generically free or has stable generic
Jordan type $(p-1,1)$. Now consider the short exact sequence $0\to
S^{(np-p,p)}{\downarrow_{E_n}}\to Q{\downarrow_{E_n}}\to
A{\downarrow_{E_n}}\to 0$. In the case $A{\downarrow_{E_n}}$ is
generically free, $S^{(np-p,p)}{\downarrow_{E_n}}$ has stable
generic Jordan type the same as $M^{(np-p,p)}{\downarrow_{E_n}}$,
which is $(1^n)$. Suppose that the stable generic Jordan type of
$A{\downarrow_{E_n}}$ is $(p-1,1)$. Using Proposition \ref{summary
of symmetric groups representation} (iii), one can figure out that
the possible stable generic Jordan types of
$S^{(np-p,p)}{\downarrow_{E_n}}$ are one of the listed below:
\begin{enumerate}
\item[(a)] $(p-1,1^{n+1})$,
\item[(b)] $(p-1,2,1^{n-1})$,
\item[(c)] $(1^n)$,
\item[(d)] $(2,1^{n-2})$.
\end{enumerate} In any case, $n-2\leq n_\mu(1)\leq n+1$.

Suppose that $p=2$. Let $\alpha=(\alpha_1,\ldots,\alpha_n)\in k^n$
be a generic point and
$u_\alpha=1+\sum^n_{i=1}\alpha_i((2i-1,2i)-1)$. Consider the map
$$M^{\mu}{\downarrow_{k\langle u_\alpha\rangle }}= M_1\oplus
P_1\xrightarrow{f} S^{(2n-1,1)}{\downarrow_{k\langle u_\alpha\rangle
}}\oplus S^{(2n)}{\downarrow_{k\langle u_\alpha\rangle }}=M_2\oplus
P_2$$ given by $f(t_{i,j})=(t_i+t_j,1)$ where $t_{i,j}$ is the
$\mu$-tabloid with $i,j$ in the second row and $t_i$ is the
$(2n-1,1)$-tabloid with $i$ in the second row, also $M_1,M_2$ are
the non-free parts of $M^\mu{\downarrow_{k\langle u_\alpha\rangle}},
S^{(2n-1,1)}{\downarrow_{k\langle u_\alpha\rangle}}\oplus
S^{(2n)}{\downarrow_{k\langle u_\alpha\rangle}}$ respectively. Note
that the Jordan types of $M_1,M_2$ are $(1^{n}),(1^2)$ respectively,
and $P_1,P_2$ are projective modules or equivalently free modules.
The kernel of the map $f$ is precisely $S^\mu{\downarrow_{k\langle
u_\alpha\rangle}}$. The module $M_1$ is $k$-spanned by $t_{2i-1,2i}$
for $1\leq i\leq n$; meanwhile, $P_1$ is $k$-spanned by those
$t_{i,j}$'s not listed before. The module $M_2$ is $k$-spanned by
both the tabloid in $S^{(2n)}$ and $\sum t_i\in
S^{(2n-1,1)}\subseteq M^{(2n-1,1)}$. Consider the map $f$ in the
stable module category, the induced map
$\widetilde{f}:M_1\twoheadrightarrow M_2$ splits and it has kernel
with Jordan type $(1^{n-2})$. So $\ker f\cong \ker
\widetilde{f}\oplus P_3$ for some projective module $P_3$. This
shows that the stable generic Jordan type of
$S^\mu{\downarrow_{E_n}}$ is $(1^{n-2})$, i.e., $n_\mu(1)=n-2$.
\end{proof}

\subsection{Proof of Theorem \ref{theorem not p by p}}

Suppose that $\mu$ is a partition satisfying Hypothesis
\ref{hypothesis}. In order to prove Theorem \ref{theorem not p by
p}, by Lemma \ref{complementary sgjt}, it suffices to show that
$V^\sharp_{E_n}\left (S^{\Phi(\mu)}\right )=V_{E_n}^\sharp(k)$. The
idea of the proof is to show that $n_{\Phi(\mu)}(r)>0$ for some
$1\leq r\leq p-1$.

\begin{proof}[Proof of Theorem \ref{theorem not p by p}] The result is obvious for Type (H3); indeed, $n_{\Phi(\mu)}(1)=1$.

\underline{Type (H1)}: For any $1\leq m< p/2$, let
$\eta(m)=(p^2-mp,mp)$ and $\lambda(m)=(p^2-mp-1,mp+1)$. Note that
$\Phi(\lambda(m))=\eta(m)$. Taking $n=p$ in Lemma \ref{generic
Jordan type case 1}, we have $n_{\eta(1)}(1)\geq p-2$. We claim that
for all $1\leq m<p/2$
$$n_{\eta(m)}(1)+n_{\eta(m)}(2)+\ldots+n_{\eta(m)}(m)\geq p-2\eqno(\ast)$$
By induction on $m$, suppose that we have the inequality as in
$(\ast)$. By Corollary \ref{decomposition for p square}, we have a
decomposition
$$M^{\eta(m+1)}\cong M^{\eta(m)}\oplus S\oplus F$$ where $F{\downarrow_{E_p}}$ is
a generically free $kE_p$-module and $S$ has a filtration with
Specht factors $S^{\lambda(m)}, S^{\eta(m+1)}$ reading from the top.
By Lemma \ref{generic jordan type of permutation modules},
$S{\downarrow_{E_p}}$ has stable generic Jordan type
$1^{\binom{p}{m+1}-\binom{p}{m}}$. Note that
$S^{\eta(m)}{\downarrow_{E_p}}, S^{\lambda(m)}{\downarrow_{E_p}}$ is
a pair of modules of complementary stable Jordan type, so
$S^{\lambda(m)}{\downarrow_{E_p}}$ has stable generic Jordan type
satisfying the inequality
$$n_{\lambda(m)}(p-m)+n_{\lambda(m)}(p-m+1)+\ldots+n_{\lambda(m)}(p-1)\geq p-2$$ The short
exact sequence $0\to S^{\eta(m+1)}{\downarrow_{E_p}}\to
S{\downarrow_{E_p}}\to S^{\lambda(m)}{\downarrow_{E_p}}\to 0$ leads
to the short exact sequences $0\to J_i\to J_a\oplus J_b\to J_j\to 0$
with $a,b\in \{0,1,p\}$. We focus on the cases where $p-m\leq j\leq
p-1$. Note that we must have either $J_p\oplus J_1$ or $J_p$ in the
middle. Using Proposition \ref{summary of symmetric groups
representation} (iii), in the first case, $i+j=1+p$ and hence $2\leq
i\leq m+1$; in the latter case, $i+j=p$ and hence $1\leq i\leq m$.
So we conclude that
$$n_{\eta(m+1)}(1)+n_{\eta(m+1)}(2)+\ldots+n_{\eta(m+1)}(m+1)\geq
p-2$$ Since $n_{\eta(m)}(i)$ is non-zero for some $1\leq i\leq m$,
$S^{\eta(m)}{\downarrow_{E_p}}$ is not generically free. So the rank
variety $V^\sharp_{E_p}(S^{\Phi(\mu)})$ is $V^\sharp_{E_p}(k)$ in
this case.

\underline{Type (H2)}: Note that $n>2$. The proof is akin to the
proof of Type (H1), except that we use Corollary \ref{decomposition
for multiple of p} instead of Corollary \ref{decomposition for p
square} . So $n_{\Phi(\mu)}(1)+n_{\Phi(\mu)}(2)\geq n-2>0$.

\underline{Type (H4)}: Consider the case where
$\Phi(\mu)=(2n-2,2)\neq (2,2)$. Note that $n>2$. Now we use Lemma
\ref{generic Jordan type case 1} for $p=2$ to deduce that
$n_{\Phi(\mu)}(1)=n-2>0$. Now take $\Phi(\mu)=(2n-4,4)\neq (4,4)$
where $n> 4$. The permutation module $M^{\Phi(\mu)}$ has a
filtration with Specht factors $S^{\lambda(i)}$ one for each $0\leq
i\leq 4$ where $\lambda(i)=(2n-i,i)$. Over $p=2$, any short exact
sequence of $kE$-modules $0\to A\to B\to C\to 0$ satisfies the
inequality $n_B(1)\leq n_A(1)+n_C(1)$. So
\[\binom{n}{2}=n_{M^{\Phi(\mu)}{\downarrow_{E_n}}}(1)\leq
\sum_{i=0}^4 n_{\lambda(i)}(1)\] Note that
$n_{\lambda(0)}(1)=n_{\lambda(1)}(1)=1$ and
$n_{\lambda(3)}(1)=n_{\lambda(2)}(1)=n-2$. Suppose that
$S^{\Phi(\mu)}{\downarrow_{E_n}}$ is generically free, i.e.,
$n_{\Phi(\mu)}(1)=0$, we deduce that $n^2-5n+4\leq 0$, i.e., $1\leq
n\leq 4$, a contradiction. So $S^{\Phi(\mu)}{\downarrow_{E_n}}$ is
not generically free, i.e.,
$V^\sharp_{E_n}(S^{\Phi(\mu)})=V^\sharp_{E_n}(k)$.
\end{proof}

\begin{eg}
Let $p$ be an odd prime and $\mu$ be the partition as in Example
\ref{an example which satisfies the main hypothesis}. Then
$V^\sharp_{E_p}(S^\mu)=V^\sharp_{E_p}(k)$ and the complexity of
$S^\mu$ is $p$.
\end{eg}

\begin{eg} Let $p=3$ and $\mu$ be a partition of 9. Suppose that
$\mu\neq (3^3)$. In the case $\widetilde{\mu}\neq \varnothing$, we
use Proposition \ref{summary of properties of varieties} (iv) and
the hook formula to calculate the rank variety
$V_{E_3}^\sharp(S^\mu)$; otherwise, we use Proposition \ref{summary
of symmetric groups representation} (v) and Theorem \ref{theorem not
p by p}. In the case $\mu=(3^3)$, the software MAGMA was used by Jon
Carlson to determine the rank variety $V_{E_3}^\sharp\left
(S^{(3^3)}\right )$ which is precisely the zero set of the radical
ideal $$\langle
{x_1}^2{x_2}^2+{x_2}^2{x_3}^2+{x_1}^2{x_3}^2\rangle$$ Note that it
fits into Theorem \ref{theorem p by p} (ii), where $\widetilde{f}=0$
and $n=1$. \medskip
\begin{center}\underline{The Variety $V_{E_3}(S^\mu)$ for the case $p=3$ and
$|\mu|=9$.} \end{center} \medskip
\begin{minipage}[b]{.5\linewidth}\centering
\begin{tabular}{cc} \toprule
$\mu$&\hspace{.5cm}$V_{E_3}(S^\mu)$\\
\midrule $(9),(1^9)$&$V_3$\\
$(81),(21^7)$&$V_3$\\
$(72),(2^21^5)$&$V_1$\\
$(63),(2^31^3)$&$V_3$\\
$(54),(2^41)$&$V_3$\\
$(71^2),(31^6)$&$V_3$\\
$(621),(321^4)$&$V_3$\\
$(531),(32^21^2)$&$\{0\}$\\
\bottomrule
\end{tabular}
\end{minipage}
\begin{minipage}[b]{.5\linewidth}\centering
\begin{tabular}{cc}
\toprule
$\mu$&\hspace{.5cm}$V_{E_3}(S^\mu)$\\
\midrule
$(4^21),(32^3)$&$V_3$\\
$(52^2),(3^21^3)$&$V_3$\\
$(432),(3^221)$&$V_3$\\
$(521^2),(421^3)$&$V_1$\\
$(431^2),(42^21)$&$V_1$\\
$(3^3)$&$V_2$\\
$(61^3),(41^5)$&$V_3$\\
$(51^4)$&$V_3$\\
\bottomrule
\end{tabular}
\end{minipage}
\medskip

\noindent In Table 1, $V_1=\res^\ast_{E_3,E_1}V_{E_1}(k)$, $V_2=V(
{x_1}^2{x_2}^2+{x_2}^2{x_3}^2+{x_1}^2{x_3}^2)$ and $V_3=V_{E_3}(k)$
where $E_1=\langle (1,2,3)\rangle\subseteq \mathfrak{S}_9$. The
subscript $i$ of $V_i$ gives the dimension, which is the complexity
of $S^\mu$.
\end{eg}

\thankyou{I thank my supervisor David Benson for his guidance and
suggestions in this paper.}

\end{document}